\documentclass[12pt]{amsart}
\usepackage{amscd}
\usepackage{amsmath}
\usepackage{amssymb}
\usepackage{units}
\usepackage[all]{xy}
\def\frk{\frak}               % font for "Fraktur"

\def\Phi{{\frk n}}
\def\Phi{{\frk N}}
\def\opn#1#2{\def#1{\operatorname{#2}}} % to make operators
\opn\chara{char} \opn\length{\ell} \opn\pd{pd} \opn\rk{rk}
\opn\projdim{proj\,dim} \opn\injdim{inj\,dim} \opn\rank{rank}
\opn\depth{depth} \opn\sdepth{sdepth} \opn\fdepth{fdepth}
\opn\grade{grade} \opn\height{height} \opn\embdim{emb\,dim}
\opn\codim{codim}  \opn\min{min} \opn\max{max}

\opn\Tr{Tr} \opn\bigrank{big\,rank}
\opn\superheight{superheight}\opn\lcm{lcm}
\opn\trdeg{tr\,deg}%\emph{
\opn\reg{reg} \opn\lreg{lreg} \opn\ini{in} \opn\lpd{lpd}
\opn\size{size}
\opn\div{div} \opn\Div{Div} \opn\cl{cl} \opn\Cl{Cl}
\opn\Spec{Spec} \opn\Supp{Supp} \opn\supp{supp} \opn\Sing{Sing}
\opn\Ass{Ass} \opn\Min{Min}
\opn\Ann{Ann} \opn\Rad{Rad} \opn\Soc{Soc}
\opn\Im{Im} \opn\Ker{Ker} \opn\Coker{Coker} \opn\Am{Am}
\opn\Hom{Hom} \opn\Tor{Tor} \opn\Ext{Ext} \opn\End{End}
\opn\Aut{Aut} \opn\id{id}  \opn\deg{deg}

\opn\nat{nat}
\opn\pff{pf}%   \pf exists already
\opn\Pf{Pf} \opn\GL{GL} \opn\SL{SL} \opn\mod{mod} \opn\ord{ord}
\opn\Gin{Gin} \opn\Hilb{Hilb}
\opn\aff{aff} \opn\con{conv} \opn\relint{relint} \opn\st{st}
\opn\lk{lk} \opn\cn{cn} \opn\core{core} \opn\vol{vol}
\opn\link{link} \opn\star{star}
\opn\gr{gr}

\def\pot#1#2{#1[\kern-0.28ex[#2]\kern-0.28ex]}

\opn\dirlim{\underrightarrow{\lim}}
\opn\inivlim{\underleftarrow{\lim}}
\let\to=\rightarrow

\def\Implies{\ifmmode\Longrightarrow \else
        \unskip${}\Longrightarrow{}$\ignorespaces\fi}
\def\implies{\ifmmode\Rightarrow \else
        \unskip${}\Rightarrow{}$\ignorespaces\fi}
\def\iff{\ifmmode\Longleftrightarrow \else
        \unskip${}\Longleftrightarrow{}$\ignorespaces\fi}

\let\:=\colon
\newtheorem{Theorem}{Theorem}[section]
\newtheorem{Lemma}[Theorem]{Lemma}
\newtheorem{Corollary}[Theorem]{Corollary}
\newtheorem{Proposition}[Theorem]{Proposition}
\newtheorem{Remark}[Theorem]{Remark}

\newtheorem{Example}[Theorem]{Example}

\let\epsilon\varepsilon
\let\phi=\varphi
\let\kappa=\varkappa
\def\qed{\ifhmode\textqed\fi
      \ifmmode\ifinner\quad\qedsymbol\else\dispqed\fi\fi}
\def\textqed{\unskip\nobreak\penalty50
       \hskip2em\hbox{}\nobreak\hfil\qedsymbol
       \parfillskip=0pt \finalhyphendemerits=0}
\def\dispqed{\rlap{\qquad\qedsymbol}}

\opn\dis{dis}
\def\pnt{{\raise0.5mm\hbox{\large\bf.}}}

\opn\Lex{Lex}

\begin{document}
\title{\bf Upper bounds of depth of  monomial  ideals}

\author{ Dorin Popescu }

\thanks{The  support from  grant ID-PCE-2011-1023 of Romanian Ministry of Education, Research and Innovation is gratefully acknowledged.}

\address{Dorin Popescu,  "Simion Stoilow" Institute of Mathematics of Romanian Academy, Research unit 5,
 P.O.Box 1-764, Bucharest 014700, Romania}
\email{dorin.popescu@imar.ro}

\maketitle
\begin{abstract}
 Let $J\subsetneq I$ be two ideals of a polynomial ring $S$ over a field, generated by square free  monomials.
 We show that some inequalities among the numbers of  square free monomials of $I\setminus J$ of different degrees give upper bounds of $\depth_SI/J$.

  \vskip 0.4 true cm
 \noindent
\end{abstract}

\section*{Introduction}

Let $S=K[x_1,\ldots,x_n]$ be the polynomial algebra in $n$ variables over a field $K$,  $d\leq t$  be two positive integers   and $I\supsetneq J$, be two  square free monomial ideals of $S$ such that $I$ is generated in degrees $\geq d$, respectively $J$ in degrees $\geq d+1$. By \cite[Theorem 3.1]{HVZ} and \cite[Lemma 1.1]{P2} $\depth_S I/J\geq d$. Let $\rho_t(I\setminus J)$  be the number of all square free monomials of degree $t$ of $I\setminus J$.
\begin{Theorem} (\cite[Theorem 2.2]{P2})\label{m0} If $\rho_d(I)>\rho_{d+1}(I\setminus J)$ then $\depth_S I/J=d$, independently of the characteristic of $K$.
\end{Theorem}
The aim of this paper is to extend this theorem.
Our Theorem \ref{m} says that $\depth_SI/J= t$ if $\depth_SI/J\geq t$ and $$\rho_{t+1}(I\setminus J)<\sum_{i=0}^{t-d} (-1)^{t-d+i} \rho_{d+i}(I\setminus J).$$
If $t=d$ then this result is stated in Theorem \ref{m0} (a previous result is given in \cite{P1}). If $t=d+1$ then the above result says that $\depth_S I/J\leq d+1$ if  $$\rho_{d+1}(I\setminus J)>\rho_{d+2}(I\setminus J)+\rho_{d}(I\setminus J).$$ A particular case with $I$  principal  is given, with a different proof, in our Proposition \ref{el}. Theorem \ref{m0} is a small step in an attempt to show Stanley's Conjecture  for some  classes of factors of square free monomial ideals (see our Remark \ref{r} for some details) and we hope that our Theorem \ref{m} will be useful in the same frame.

\section{Upper bounds of depth}
The aim of this section is to show the extension of Theorem \ref{m0}  stated in the introduction. We start with a particular case.
\begin{Proposition}\label{el}  Suppose that $I$ is  generated by a square free monomial $f$ of degree $d$, and $s=\rho_{d+1}(I\setminus J)>\rho_{d+2}(I\setminus J)+1$. Then $\depth_S I/J= d+1$.
\end{Proposition}
\begin{proof} First suppose that $q=\rho_{d+2}(I\setminus J)>0$. Let $g\in I\setminus J$ be a square free monomial of degree $d+2$. Renumbering the variables $x$ we may suppose that $I$ is generated by $f=x_1\cdots x_d$ and $g=fx_{d+1}x_{d+2}$. Since $g\not \in J$ we see that $b_1=fx_{d+1}$,
$b_2=fx_{d+2}$ are not in $J$. Again renumbering $x$  we may suppose that $b_i=fx_{d+i}$, $i\in [s]$ are all the square free monomials of degree $d+1$ from $I\setminus J$. Set $T=(b_3,\ldots,b_s)$ (by hypothesis $s\geq 3$). In the exact sequence
$$0\to T/T\cap J\to I/J\to I/(T+J)\to 0$$
we see that the left end has depth $d+1$  by Theorem \ref{m0} since $T\cap J$ is generated in degree $\geq d+2$ and $\rho_{d+1}(T)=s-2>q-1=\rho_{d+2}(T\setminus T\cap J)$.
On the other hand, $(T+J):f=(x_{d+3},\ldots,x_n)$ because $b_1,b_2,g\not\in T+J$. It follows that $\depth_S I/(T+J)=d+2$ and so the Depth Lemma says that $\depth_S I/J=d+1$.

Now suppose that $q=0$. As above we may assume that $b_i=fx_{d+i}$, $i\in [s]$ are the square free monomials of degree $d+1$ of $I\setminus J$. Then $J:f=(x_{d+s+1},\dots,x_n)+L$, where $L$ is the Veronese ideal generated by all square free monomials of degree $2$ in $x_{d+1},\ldots,x_{d+s}$. It follows that $I/J\cong K[x_1,\ldots, x_{d+s}]/L$ which has depth $d+1$.
\end{proof}

 Next  we present some details on the Koszul homology (see \cite{BH}) which we need for the proof of our main result. Let $\partial_i:K_i(x;I/J)\to K_{i-1}(x;I/J)$,  $K_i(x;I/J)\cong S^{{n\choose i}}$, $i\in [n]$ be the Koszul derivation given by
 $$\partial_i(e_{j_1}\wedge \ldots \wedge e_{j_i})=\sum_{k=1}^i (-1)^{k+1}x_{j_k} e_{j_1}\wedge \ldots \wedge e_{j_{k-1}}\wedge e_{j_{k+1}}\wedge\ldots \wedge e_{j_i}.$$
  Fix $0\leq i<n-d$. Let $f_1,\ldots,f_r$, $r=\rho_{d+i}(I\setminus J)$ be all square free monomials of degree $d+i$ from $I\setminus J$ and $b_1,\cdots,b_s$, $s=\rho_{d+i+1}(I\setminus J)$ be all square free monomials of degree $d+i+1$ from $I\setminus J$.   Let $\supp f_i=\{j\in [n]: x_j|f_i\}$,\ $e_{\sigma_i}=\wedge_{j\in ([n]\setminus \supp f_i)}\ e_j$ and $\supp b_k=\{j\in [n]: x_j|b_k\}$,\ $e_{\tau_k}=\wedge_{j\in ([n]\setminus \supp b_k)}\ e_j$. We consider the  element
  $z=\sum_{q=1}^ry_qf_q e_{\sigma_q}$
  of $K_{n-d-i}(x;I/J)$, where $y_q\in K$. Then $$\partial_{n-d-i}(z)=\sum_{k=1}^s(\sum_{q\in [r]} \epsilon_{kq}y_q)b_k e_{\tau_k},$$
 where   $ \epsilon_{kq}\in \{1,-1\}$ if $f_q|b_k$, otherwise  $ \epsilon_{kq}=0$. Thus $\partial_{n-d-i}(z)=0$ if and only if $\sum_{q\in [r]} \epsilon_{kq}y_q=0$ for all $k\in [s]$, that is $y=(y_1,\ldots,y_r)$ is in the kernel of the linear map
 $h_{n-d-i}:K^r\to K^s$ given by the matrix $\epsilon_{kq}$.

 Now we will see when $z\in \Im \partial_{n-d-i+1}$. Since the Koszul derivation is a graded map we note that $z\in \Im \partial_{n-d-i+1}$ if and only if $z=\partial_{n-d-i+1}(w)$ for a $w=\sum_{p=1}^c u_pg_p e_{\nu_p}$, where  $c=\rho_{d+i-1}(I\setminus J)$, $u_p\in K$, $g_1,\ldots,g_c$ are all square free monomials of degree $d+i-1$ from $I\setminus J$ and $e_{\nu_p}=\wedge_{j\in ([n]\setminus \supp g_p)}\ e_j$. It follows $$z=\sum_{q=1}^r(\sum_{p\in [c]} \gamma_{qp}u_p)f_q e_{\sigma_q},$$
where   $ \gamma_{qp}\in \{1,-1\}$ if $g_p|f_q$, otherwise  $ \gamma_{qp}=0$. Thus  $z\in \Im \partial_{n-d-i+1}$ if and only if $y$ belongs to the image of the linear map $h_{n-d-i+1}:K^c\to K^r$ given by the matrix $\gamma_{qp}$. When $i=0$ we have $h_{n-d-i+1}=0$.

Note that $\Im h_{n-d-i+1}\subset \Ker h_{n-d-i}$ and the inclusion is strict if and only if there exists $y\in K^r$ such that $z$ induces a nonzero element in $H_{n-d-i}(x;I/J)$. This implies $\depth_S I/J\leq d+i$ by \cite[Theorem 1.6.17]{BH}. If  $\depth I/J> d+i$ then $\Im \partial_{n-d-i+1}=\Ker \partial_{n-d-i}$ and it follows $\Im h_{n-d-i+1}= \Ker h_{n-d-i}$.

\begin{Lemma}\label{key}  Let $0\leq i<n-d$, then the following statements hold independently of the characteristic of $K$.
\begin{enumerate}
\item{} the complex $K^c\xrightarrow{h_{n-d-i+1}} K^r\xrightarrow{h_{n-d-i}} K^s$ is exact if   $\depth I/J> d+i$,
\item {} if $\depth_S I/J> d+i$ then $r=\rank h_{n-d-i+1}+\rank h_{n-d-i}$,
\item{} if $r>\rank h_{n-d-i+1}+\rank h_{n-d-i}$ then $\depth_S I/J\leq d+i.$
\end{enumerate}
\end{Lemma}
\begin{proof} The first statement follows from above and the second one is only a consequence.  If $r>\rank h_{n-d-i+1}+\rank h_{n-d-i}$ then $\Im h_{n-d-i+1}\subsetneq \Ker h_{n-d-i}$ and the last statement follows also from above.
\end{proof}

\begin{Theorem} \label{m} Let $d\leq t\leq n$ be two integers and  set
$$\alpha_j=\sum_{i=0}^{j-d} (-1)^{j-d+i} \rho_{d+i}(I\setminus J),$$
for $d\leq j\leq t$.
Suppose that $\depth_SI/J\geq t$ and $\rho_{t+1}(I\setminus J)<\alpha_t$.
Then $\depth_SI/J=t$ independently of the characteristic of $K$.
\end{Theorem}
\begin{proof} We have $\alpha_j=\rho_j(I\setminus J)-\alpha_{j-1}$ for $d<j\leq t$. By Lemma \ref{key} (2) we get $h_{n-d}$ injective and $\rho_{d+i}(I\setminus J)=\rank h_{n-d-i+1}+\rank h_{n-d-i}$ for $0<i<t-d.$ It follows that $\rho_{d}(I\setminus J)=\rank h_{n-d}=\alpha_d$,
$\rho_{d+1}(I\setminus J)=\rank h_{n-d}+\rank h_{n-d-1}=\rho_{d}(I\setminus J)+\rank h_{n-d-1}$, and so $\rank h_{n-d-1}=\alpha_{d+1}$.
By recurrence we get $\rank h_{n-t+1}=\alpha_{t-1}$.
Clearly, $\rank h_{n-t}\leq \rho_{t+1}(I\setminus J)$. By hypothesis,  $\rho_{t+1}(I\setminus J)<\alpha_t=\rho_t(I\setminus J)-\alpha_{t-1}$. It follows that $\rank h_{n-t}<\rho_t(I\setminus J)-\alpha_{t-1}=\rho_t(I\setminus J)-\rank h_{n-t+1}$ which gives
 $\depth_SI/J=t$ by Lemma \ref{key} (3).
\end{proof}

Next example shows that the above theorem is tight.
\begin{Example} {\em Let $n=4$, $I=(x_1,x_3)$, $J=(x_1x_4)$. Note that $x_1x_2,x_1x_3,x_2x_3,x_3x_4$ are all square free monomials of degree $2$ from $I\setminus J$ and $x_1x_2x_3,x_2x_3x_4$ are all square free monomials of degree $3$ from $I\setminus J$. Thus $\rho_2(I\setminus J)=4=\rho_1(I)+\rho_3(I\setminus J)$, but   $\depth_SI/J= 3$.
On the other hand, taking $J'=J+(x_2x_3x_4)$ we see that $\depth_SI/J'=2$ which is given also by Theorem \ref{m} since   $\rho_3(I\setminus J')=1$ and we have $\rho_2(I\setminus J')=4>3=\rho_1(I)+\rho_3(I\setminus J')$.}
\end{Example}

\begin{Corollary}\label{c} Suppose that $\depth_S I/J\geq d+2$. Then $\rho_d(I)\leq \rho_{d+1}(I\setminus J)\leq \rho_d(I)+
\rho_{d+2}(I\setminus J)$. Moreover, if  $\rho_{d+2}(I\setminus J)=0$ then $\rho_d(I)=\rho_{d+1}(I\setminus J)$.
\end{Corollary}

\begin{Remark}\label{r} {\em Consider the poset $P_{I\setminus J}$  of all square free monomials of $I\setminus J$ (a finite set) with the order given by the divisibility. Let ${\mathcal P}$ be a partition of  $P_{I\setminus J}$ in intervals $[u,v]=\{w\in  P_{I\setminus J}: u|w, w|v\}$, let us say   $P_{I\setminus J}=\cup_i [u_i,v_i]$, the union being disjoint.
Define $\sdepth {\mathcal P}=\min_i\deg v_i$ and $\sdepth_SI/J=\max_{\mathcal P} \sdepth {\mathcal P}$, where ${\mathcal P}$ runs in the set of all partitions of $P_{I\setminus J}$. This is the Stanley depth of $I/J$, in the idea of  \cite{HVZ} (see  also \cite{S}).
 Stanley's Conjecture says that  $\sdepth_S I/J\geq \depth_S I/J$. In the above corollary $\rho_{d+2}(I\setminus J)=0$ implies
 $\sdepth_SI/J\leq d+1$ and so $\depth_SI/J\leq d+1$ if Stanley's Conjecture holds. This shows the weakness of the above corollary,
 which accepts the possibility to have $\depth_S I/J\geq d+2$ when $\rho_d(I)=\rho_{d+1}(I\setminus J)$.
}
\end{Remark}


\vskip 0.5 cm



\begin{thebibliography}{99}

\bibitem{BH} W.\ Bruns and J. Herzog, {\em Cohen-Macaulay rings}, Revised edition. Cambridge University Press (1998).
\bibitem{HVZ} J.\ Herzog, M.\ Vladoiu, X.\ Zheng, {\em How to compute the Stanley depth of a monomial ideal,}  J.  Algebra, 322 (2009), 3151-3169.

\bibitem{P1} D.\ Popescu, {\em  Depth and minimal number of generators of  square free monomial ideals}, An. St. Univ. Ovidius, Constanta, 19(3), (2011), 163-166, arXiv:AC/1107.2621.
\bibitem{P2} D.\ Popescu, {\em Depth of  factors of square free  monomial  ideals}, arXiv:AC/1110.1963.

\bibitem{S} R.\ P.\ Stanley, {\em Linear Diophantine equations and local cohomology,} Invent. Math. 68 (1982) 175-193.

\end{thebibliography}
\end{document}